\documentclass[12pt]{amsart}
\usepackage{amsfonts}
\usepackage{amscd}
\usepackage{amssymb}
\usepackage{graphics}

\usepackage{amsmath}

\usepackage{hyperref}
\hypersetup{colorlinks,citecolor=blue}

\newcommand{\BZ}{{\mathbb{Z}}}
\newcommand{\BN}{{\mathbb{N}}}

\newcommand{\BA}{{\mathbb{A}}}

\newcommand{\BK}{{\mathbb{K}}}

\newcommand{\OO}{{\mathcal{O}}}

\newcommand{\gb}{\beta}

\newcommand{\gC}{\Gamma}

\newcommand{\gO}{\Omega}

\newcommand{\ga}{\alpha}

\newcommand{\WAP}{\text{WAP}}

\newcommand{\End}{\text{End}}

\newcommand{\SL}{\text{SL}}
\newcommand{\GL}{\text{GL}}

\newcommand{\SO}{\text{SO}}
\newcommand{\Iso}{\text{Iso}}

\newcommand{\id}[1]{id_{#1}}

\newtheorem{prop}{Proposition}[section]
\newtheorem{thm}[prop]{Theorem}
\newtheorem{lem}[prop]{Lemma}
\newtheorem{cor}[prop]{Corollary}

\theoremstyle{definition}

\newtheorem{defn}[prop]{Definition}
\newtheorem{rem}[prop]{Remark}

\newtheorem{exam}[prop]{Example}

\catcode`\@=11
\long\def\@savemarbox#1#2{\global\setbox#1\vtop{\hsize\marginparwidth 
  \@parboxrestore\tiny\raggedright #2}}
\catcode`\@=12

\newcommand{\act}{\curvearrowright}

\newcommand{\bbR}{{\mathbb R}}

\newcommand{\bbC}{{\mathbb C}}

\newcommand{\Prob}{\operatorname{Prob}}

\newcommand{\conv}{\operatorname{conv}}

\newcommand{\acts}{\curvearrowright}

\newcommand{\form}[2]{{\langle #1,#2 \rangle}}

\begin{document}

\title{Equicontinuous actions of semisimple groups}

\date{July  26, 2015}

\author{Uri Bader and Tsachik Gelander}\thanks{The research was partly supported by the ERC and the ISF}

\begin{abstract}
We study equicontinuous actions of semisimple groups and some generalizations.
We prove that any such action is universally closed, and in particular proper.
We derive various applications, both old and new, including closedness of continuous homomorphisms, nonexistence of weaker topologies, metric ergodicity of transitive actions and vanishing of matrix coefficients for reflexive (more generally: WAP) representations. 
\end{abstract}
\maketitle
\tableofcontents

\section{Historical prelude and introduction}

We begin by presenting some of the history of the ideas around the so called ``Mautner phenomenon", Moore ergodicity theorem, the Howe--Moore theorem
and related topics.
This mixture of ideas and techniques fascinatingly relates ergodic theory, representation theory, topological group theory and metric geometry.

The term ``Mautner phenomenon" is used to describe the idea behind the following easy lemma and its various generalizations.

\begin{lem} \label{lem:maut}
Let $G$ be a topological group acting isometrically on a metric space $(X,d)$ such that the homomorphism $G\to \Iso(X)$ is continuous.
Assume that $(a_n)$ is a sequence in $G$, and $u\in G$ is an element such that $\lim_n u^{a_n}=e$ in $G$.
If $x\in X$ satisfies $\lim_{n} a_nx=x$ then $ux=x$.
\end{lem}

The beautiful one line proof is given by
\[ d(ux,x)=\lim_n d(ua_n^{-1}x,a_n^{-1}x)=\lim_n d(u^{a_n}x,x)=d(\lim_n u^{a_n}x,x)=d(ex,x)=0.\]

It seems to us that the first documented use of this idea is from 1950 by Segal and von Neumann, see \cite[Lemma 1]{SVN}.
Mautner used the ``Mautner phenomenon" in his 1957 paper \cite{Mautner} to establish the ergodicity of the geodesic flow on finite volume locally symmetric spaces,
applying it to the dynamics of the associated semisimple Lie group on a corresponding unitary representation.
Mautner's result dramatically generalizes Hopf's result \cite{Hopf} which treats the rank one case, by a beautiful geometric argument.
The powerful idea of using unitary representations for the study of ergodicity of the geodesic flow is due to Gelfand and Fomin \cite{GF}
who used it for the special case of manifolds of dimension 2 and 3, using the explicit classification of the irreducible representation of $SL_2(\bbR)$ and $\SL_2(\bbC)$.

Mautner's work was shortly after generalized by Moore who showed in \cite{Moore} 
that for every ergodic probability measure preserving action of a simple Lie group, every one parameter subgroup is mixing.
Moore's work, in turn, was generalized by Zimmer and Howe who obtained  independently the following well known theorem.

\begin{thm}[{\cite[Theorem~5.2]{Zimmer}},{\cite[Theorem~5.1]{HM}}]
In a unitary representation of a simple Lie group which has no non-zero invariant vectors, the matrix coefficients tend to 0 at infinity. 
\end{thm}

Zimmer's proof relies on a theorem proved independently in \cite[Theorem~1]{Sherman} and \cite[Theorem~8]{Moore70}, describing the restriction of a unitary representation of a simple Lie group to a maximal split torus.
The approach of Howe--Moore is based on the Mautner phenomenon and works for uniformly bounded Hilbert representations and over any local field.
A main extra technical ingredient there is the reduction of the statement to a statement about $SL_2$, via the Jacobson--Marozov theorem.

It is interesting to note that the tone in \cite{HM} is apologetic. Before stating their Theorem~5.1 they write:
``Our only excuse for including this in view of the fact that more precise results 
are known (but involving some effort and machinery) is that it is simple and 
direct and already contains useful information."
The ``more precise results" they refer to are results of Cowling and Wallach regarding 
the asymptotic behavior of matrix coefficients of irreducible uniformly bounded representations,
e.g that the
matrix coefficients are in $L^p$ for some $p$,
see \cite{Cowling} for the archimedean case and \cite{BW} for the general case.

The Howe--Moore theorem was soon generalized by Veech who obtained in \cite{veech} a similar theorem applicable for any uniformly bounded representation on any reflexive Banach space.
Veech's result is again an elaboration on the Mautner phenomenon, now applied in the context of the WAP (Weakly Almost Periodic) compactification of a semisimple group.
The WAP compactification of a group is a universal semi-topological, semi-group compactification which was studied by de Leeuw and Glicksberg, following a fundamental work of Eberlein, Grothendieck and others on WAP functions on groups.
Matrix coefficients of uniformly bounded reflexive representations are WAP functions,
thus the WAP compactificqation gathers information on all reflexive representations.
It is interesting to note that by the main theorem of \cite{DFJP}, the converse is also true: every WAP function is a matrix coefficient of a reflexive representation, see \cite{Me03} (check also \cite{kaijser}).
In \cite{Cowling79} Cowling proves a similar, though slightly weaker theorem: he considers the Fourier--Stieltjes compactification of a semisimple Lie group
and proves a parallel result to Veech's.
His proof is representation theoretic.
Veech's result was reproduced and put in a conceptual context by Ellis and Nerurkar in \cite{EN}.
We find the papers \cite{veech} and \cite{EN} very appealing and we are surprised how little attention they got.
For example, when we wrote \cite{Bader-Furman-Gelander-Monod}, together with Furman and Monod, we included an appendix, based on an observation of Shalom,
which reproduced the so called Howe--Moore Theorem in the restrictive context of super-reflexive Banach spaces, see \cite[Theorem~9.1]{Bader-Furman-Gelander-Monod}.
No one ever, till this day, had brought to our attention that this result is subsumed in \cite{veech}.
After the publication of \cite{Bader-Furman-Gelander-Monod} we realized that Theorem~9.1, with the same proof essentially, could be easily generalized to all reflexive representations, with the aid of one extra ingredient: the Ryll-Nardzewski Theorem.
This already brought us to consider the WAP compactification. 
Considering the latter, we soon came across \cite{veech}.

From the late 70's to this day the Howe--Moore theorem stands as a corner stone in the analytic theory of semisimple groups.
For example, it is a key ingredient in Margulis' proof of his celebrated super-rigidity theorem.
It has numerous other applications.
Let us mention here only two very easy ones.
The first one is the easy proof of an older result due to Tits and Prasad \cite{Prasad}:
every proper open subgroup of a simple group over a local field is compact.
Indeed, this follows at once by considering the matrix coefficients of the corresponding quasi-regular representation.
The second application, given in  \cite[Theorem~7.3]{Zimmer}, could be seen as a strengthening of simplicity: for a simple group over a local field, 
every non-trivial continuous homomorphism with dense image into a locally compact group is an open bijection.
This follows by considering the regular representation of the target group.
Let us elaborate on this last application.

Already in 1966, Omori had proved that every continuous homomorphism from a semisimple Lie group (with finite center) into any first countable topological group has a closed image, \cite{omori}.
This generalizes previous results of van Est (proving a similar theorem with Lie group targets, \cite{vE}) and Goto (same with locally compact targets, \cite{Goto48} and \cite{GY}). 
See \cite{Goto73} for a general discussion.
Groups satisfying the property that every injective image in a topological group is closed are called ``absolutely closed". Their group topology is called ``minimal".
For a recent extensive survey on the subject, see \cite{DM}.
As mentioned above, Zimmer reproved Goto's result, as a corollary of the Howe--Moore theorem.
In \cite[Theorem~2.1]{HM} it is also shown that the image of a homeomorphism of a simple Lie group in the unitary group of a Hilbert space is closed. It seems that the three authors were unaware of Omori's result from 1966, \cite{omori}.
A more systematic attempt to relate the notion of coarse group topologies and the study of matrix coefficients is given in \cite{Mayer} for connected Lie groups.
It seems that for general locally compact groups there is still unexplored ground in this direction.
Here we observe that Omori's theorem (even without the first countability assumption) is an immediate application of the following theorem.

\begin{thm} \label{demo}
Let $G$ be a simple group over a local field.
Assume $G$ acts equicontinuously and without fixed points on a uniform space.
Then every orbit is closed and all the point stabilizers are compact.
\end{thm}

Indeed, given a continuous homomorphism into a topological group, $G\to H$, endowing $H$ with its left uniform structure and consider the left action of $G$ on $H$ we get that the orbit of $e\in H$, that is the image of $G$, is closed.
We get the following extension of Omori's theorem.

\begin{cor}
A simple group over a local field is absolutely closed.
\end{cor}

Our main contribution in this paper is the formulation of Theorem~\ref{mainthm}, which is a simultaneous generalization of Veech's theorem and Theorem~\ref{demo} above.
Theorem~\ref{mainthm} discusses an action of a group on a space with two compatible structures: a uniform structure and a compatible (typically weaker) topology. The reader familiar with the classical Howe--Moore theorem should have in mind the two compatible structures on the unit ball of a Banach space: the norm and the weak topologies.
We will not formulate Theorem~\ref{mainthm} in this introduction, due to its technical nature. 
Instead we choose to demonstrate it by proving a toy case, which already contains most of the ideas of the proof.
For a generalization of the next theorem, see Theorem~\ref{thm:iso}.

\begin{thm}[toy case]\label{thm:baby}
Assume the group $G=\SL_2(\bbR)$ is acting continuously by isometries on a metric space $X$. Let $x_0\in X$ be a point. Then either $x_0$ is a global fixed point or its stabilizer group is compact.
\end{thm}

Note that every metric space has a canonical uniform structure, and every isometric action is equicontinuous,
thus Theorem~\ref{thm:baby} is an immediate corollary of Theorem~\ref{demo}.
To the best of our knowledge, even this very specific result has not been formulated in the past literature.
Here is a complete proof:

\begin{proof}
Suppose that the stabilizer group $G_{x_0}$ is not compact and let $g_n\in G_{x_0}$ be a sequence tending to infinity. Write $g_n=k_n a_n k'_n$ with $k_n,k_n'\in \SO(2)$ and 
\[ a_n=
 \left( \begin{array}{cc}
\ga_n & 0 \\
0 & \ga_n^{-1} 
\end{array} \right)
\]
with $\ga_n\to \infty$.
Up to replacing $g_n$ by a subsequance we may suppose that $k_n\to k$ and $k_n'\to k'$. Set $x_1=k'\cdot x_0$ and $x_2=k^{-1}\cdot x_0$.

Let $m:\BN\to\BN$ be a function tending sufficiently fast to $\infty$ so that $\gb_n:=\frac{\ga_{m(n)}}{\ga_n}\to\infty$, and set $b_n=a_{m(n)}a_n^{-1}$, that is
\[ b_n=
 \left( \begin{array}{cc}
\gb_n & 0 \\
0 & \gb_n^{-1} 
\end{array} \right).
\]
Since the action $G\act X$ is continuous by isometries, and since $a_n\cdot x_1\to x_2$ and $a_n^{-1}\cdot x_2\to x_1$, we see that $b_n\cdot x_2\to x_2$, as well as $b_n^{-1}\cdot x_2\to x_2$.
Note that for
\[ u_+(t)=
 \left( \begin{array}{cc}
1 & t \\
0 & 1 
\end{array} \right)~ \text{and}~
u_-(t)=
 \left( \begin{array}{cc}
1 & 0 \\
t & 1 
\end{array} \right)
\]
we have $\lim_{n\to\infty} b_n^{-1}u_+(t)b_n=\lim_{n\to\infty} u_+({t\over \gb_n^2})=e$, thus by Lemma~\ref{lem:maut}, $u_+(t)x_2=x_2$.
Similarly, we have that $\lim_{n\to\infty} b_nu_-(t)b_n^{-1}=e$, hence $u_-(t)x_2=x_2$.
Since $G$ is generated by $u_+(t)$ and $u_-(t)$ we deduce that $G=G_{x_2}$. Finally, it follows that $x_2=x_0$, hence $x_0$ is a global fixed point.
\end{proof}

In the above proof, note the role of the sequence $(b_n)$ which zig-zags between the accumulation points $x_1$ and $x_2$.
This is the only novel ingredient in the following short list which summarizes the main ingredients appearing in the above proof:
\begin{itemize}
\item the $KAK$ decomposition,
\item the Mautner phenomenon,
\item a zig-zaging argument, and
\item the generation of $G$ by opposite unipotents.
\end{itemize}
In the proof of Theorem~\ref{mainthm}, the zig-zaging argument will come about by the use of Lemma~\ref{difference}.

Finally note that for the proof we don't need the precise structure of the group $G$, only the properties allowing us to use the ingredients above.
This observation was used before by several authors, generalizing the Howe--Moore theorem to various non-algebraic groups, notably groups acting on trees,
see \cite{LM}.
We will prove our main theorems for groups having these appropriate properties, which we call ``quasi-semisimple", or shortly qss groups.
A similar axiomatic approach is taken in the recent preprint \cite{ciob}.




\subsection{The structure of the paper}
The first half of the paper is devoted to the formulation and proof of the main Theorem~\ref{mainthm} about equicontinuous actions of semisimple (and, more generally, qss) groups, including a presentation of the basic notions
and necessary background. The second half (\S\ref{sec:closed-image} and further) is dedicated to various applications of Theorem \ref{mainthm}.
The very last section deals with some further generalizations of our axiomatic scheme.

\S\ref{preli} summarizes some well-known background material.
In \S\ref{qss} we present the class of quasi-semisimple groups, the class to which we apply our main theorems
proven in \S\ref{sec:main}.
In \S\ref{sec:topologies} we prove the inexistence of weak topologies on qss groups improving old results of Omori and Goto.
In \S\ref{sec:ME} we establish new results concerning metric ergodicity of analytic semisimple groups and their lattices.
In \S\ref{sec:monoid} we review the theory of monoid compactifications of a group
and reprove (and slightly extend) Veech's theorem.
We apply these results in \S\ref{sec:mixing} in order to gain information about Banach representations.
In \S\ref{sec:Banach} we reprove the results of \S\ref{sec:mixing} directly from Theorem~\ref{mainthm},
for the benefit of the reader who wishes to avoid the abstract setting of \S\ref{sec:monoid}. 
In \S\ref{sec:hqss} we discuss the extension of our results to the class of hereditary quasi-semisimple groups.

\subsection{Acknowledgment}

We thank Amos Nevo for an enlightening conversation regarding the history and Gil Goffer and the excellent referee for many remarks and suggestions for improvements.
We thank Michael Megrelishvili for some helpful suggestions and references.
We are grateful to Meny Shlossberg for spotting a gap in the proof of Theorem~\ref{thm:image} in an earlier version of this paper.

%
%
%

%
%

\section{Preliminaries} \label{preli}

\subsection{On nets convergence}

Recall that a net in a topological space is a map to the space from a directed set, where a directed set is a pre-ordered set in which every two elements have an upper bound.
Typically we denote a directed set by the symbol $(\alpha)$ where $\alpha$ denotes a generic element in the directed set, and for a net in the topological space $X$ we use symbols as $(x_\alpha)$, representing the map $\alpha \mapsto x_\alpha$.

The net $x_\alpha$ converges to $x$, to be denoted $x_\alpha \to x$, if for every neighborhood $U$ of $x$ there exists $\alpha_0\in (\alpha)$ such that for every $\alpha \geq \alpha_0$, $x_\alpha \in U$.

A net $(x_\beta)$ is said to be a subnet of the net $(x_\alpha)$ if it is obtained as the composition of an order preserving cofinal map $(\beta)\to (\alpha)$
with the map $\alpha \mapsto x_\alpha$.
It is well known and easy to check that a net converges if and only if all of its subnets converge and to the same point.
Less well known is the fact that every net which majorizes a subnet of a converging net converges as well.

\begin{lem} \label{subnetmajoration}
Let $n:(\alpha)\to X$ be a net converging to $x$. Let $f:(\beta)\to (\alpha)$ be an order preserving cofinal map.
Let $f':(\beta)\to (\alpha)$ be another map, satisfying for every $\gb$, $f'(\gb)\geq f(\gb)$.
Then the net $n\circ f'$ converges to $x$.
\end{lem}

\begin{proof}
Fixing a neighborhood $U$ of $x$ we need to show that 
there exists $\beta_0 \in (\beta)$ such that for every $\beta \geq \beta_0$, $x_{f'(\beta)} \in U$.
Indeed, by the convergence of the net $x_\alpha$ there exists $\alpha_0\in (\alpha)$ such that for every $\alpha \geq \alpha_0$, $x_\alpha \in U$,
and by the cofinality of $f$, there exists $\beta_0\in (\beta)$ such that $f(\beta_0)\geq \alpha_0$.
Then for every $\beta \geq \beta_0$, $f'(\beta)\geq f(\beta) \geq \alpha_0$ implies $x_{f'(\beta)}\in U$.
\end{proof}

In a locally compact space $X$, a net is said to converge to infinity if for every compact subset $K$ there exists  $\alpha_0$ such that for every $\alpha \geq \alpha_0$, $x_\alpha \notin K$.
The following technical lemma will be of use.

\begin{lem}[Zig-Zag lemma] \label{difference}
Let $G$ be a locally compact group acting on a topological space $X$. 
Let $g_\alpha$ be a net in $G$ converging to infinity and assume that for some $x,y\in X$, the net $(g_\alpha x)$ converges to $y$ in $X$.
Then there exists a directed set $(\beta)$ and two nets $n,n':(\beta)\to G$ satisfying
$n(\beta)x\to y$ and $n'(\beta)x \to y$ in $X$ and $n(\beta)^{-1}n'(\beta) \to \infty$ in $G$.
\end{lem}

\begin{proof}
We let $\mathcal{C}$ be the directed set of compact subsets of $G$, ordered by reversed inclusion, and set $(\beta)=(\alpha)\times \mathcal{C}$
endowed with the product order.
We let $f:(\beta)\to (\alpha)$ be the projection on the first variable.
This is obviously an order preserving cofinal map.
For every $(\alpha_0,K)\in (\beta)$ we use the fact that the subnet $(g_\alpha)_{\alpha\geq \alpha_0}$ converges to infinity in $G$
to find an element $\alpha_1\geq \alpha_0$ satisfying $g_{\alpha_1} \notin g_{\alpha_0}K$.
We denote $\alpha_1=f'(\alpha_0,K)$.
The lemma now follows from Lemma~\ref{subnetmajoration}, setting $n(\beta)=g_{f(\beta)}$ and $n'(\beta)=g_{f'(\beta)}$.
\end{proof}

\subsection{Uniform structures and compatible topologies}

Recall that a uniform structure on a set $X$ is a symmetric filter $S$ of reflexive relations on $X$, such that for every $U\in S$ there is $U'\in S$ with $U'U'\subset U$. Here 
$$ 
 U_1U_2=\{(u_1,u_2):\exists u_3, (u_1,u_3)\in U_1,(u_3,u_2)\in U_2\}.
$$

Let  $(X,S)$ be a uniform space.

\begin{defn} 
We will say that a topology $T$ on $X$ is {\it $S$-compatible} if for every $V\in T$ and a point $y\in V$, there exists $y\in V'\in T$ and $U\in S$ such that
$UV'\subset V$, where 
$$
 UV'=\{v~|~\exists v' \in V',~(v,v')\in U \}.
$$
\end{defn}

We shall denote by $T_S$ the {\it $S$-topology} on $X$, i.e. the topology generated by the sets 
$$
 U(x):= U\{x\},~x\in X,U\in S.
$$
Obviously, we have:

\begin{exam}
The $S$-topology $T_S$ is $S$-compatible.
\end{exam}

A topological group action on a topological space $G\act (X,T)$ is said to be {\it jointly continuous} or simply {\it continuous} if the action map $G\times X\to X$ is continuous as a function of two variables.

\begin{exam} \label{action-structure}
Given an action of a topological group $G$ on a set $X$ we define the {\em action uniform structure} $S_G$ on $X$ to be the uniform structure generated by the images of the sets
$U\times X$ under the map 
$$
 G\times X \to X\times X,~(g,x)\mapsto (x,gx),
$$ 
where $U$ runs over the identity neighborhoods in $G$.
A topology $T$ on $X$ is $S_G$-compatible if and only if the action of $G$ on $(X,T)$ is continuous.
\end{exam}

A group action on a uniform space $G\act X$ is said to be {\it equicontinuous} (or sometimes {\it uniformly continuous}) if for every $U \in S$, also the set $\{(u,v)~|~\forall g\in G,~(gu,gv)\in U\}$ is in $S$.
This means that $S$ has a basis consisting of $G$-invariant uniformities.

\begin{exam} \label{uniform-group}
For a topological group $G$, setting $X=G$, the right regular action defines a uniform structure on $G$, as in Example~\ref{action-structure}.
This structure is called the {\em left} uniform structure. 
Note that the {\em left} regular action is equicontinuous with respect to that structure.
\end{exam}

\begin{lem} \label{uniformquotient}
Assume $G$ acts on $(X,S)$ uniformly. Denote by $X/G$ the space of orbits and denote by $\pi:X \to X/G$ the natural quotient map.
Then the collection $\{(\pi\times\pi)(U)~|~U\in S\}$ defines a uniform structure on $X/G$, to be denoted $\pi_*S$,
and the associated topology on $X/G$, $T_{\pi_*S}$ coincides with the quotient topology $\pi_*T_S$. 
\end{lem}

\begin{proof}
Left to the reader.
\end{proof}

\begin{lem} \label{Joint}
An equicontinuous action of a topological group is (jointly) continuous with respect to the $S$-topology if (and only if) the orbit maps are continuous.
\end{lem}

\begin{proof}
For any $y\in X$ and a neighborhood of the form $U(y)$ associated with a uniformity $U\in S$,
there exists a $G$-invariant uniformity $U'$ such that $U'U'\subset U$.
For any $(g,x)$ with $gx=y$, 
let $\gO\subset G$ be the pre-image of $U'(y)$ under the $x$-orbit map.
Then $\gO\times U'(x)$ is a neighborhood of $(g,x)$ in $G\times X$ whose image under the action map is contained in $U(y)$. Indeed, for $(g',x')\in \gO\times U'(x)$, $(x',x)\in U'$ implies that $(g'x',g'x)\in U'$ which together with $(g'x,y)\in U'$ gives $(g'x',y)\in U$.
%
%
%
%
\end{proof}

\begin{lem} \label{CI}
Let $G\act (X,S)$ be an equicontinuous action. Let $T$ be an $S$-compatible topology on $X$.
Let $(\alpha)$ be a directed set.
Assume that $x_\alpha$ is a $T_S$-converging net in $X$ with $T_S\text{-}\lim x_\alpha=x$,
and that $g_\alpha$ is a net in $G$.
Then $T\text{-}\lim g_\alpha x_\alpha$ exists if and only if $T\text{-}\lim g_\alpha x$ exists, in which case they are equal.
\end{lem}

\begin{proof}
Let $x_\ga'$ be an arbitrary net in $X$ which $T_S$-converges to $x$.
Suppose that $T\text{-}\lim g_\alpha x_\alpha$ exists and denote it by $y$.
Let $V\in T$ be a neighborhood of $y$.
We will show that there exists $\alpha_0$ such that $\alpha\geq \alpha_0$ implies $g_\alpha x'_\alpha \in V$.
Fix $V'\in T$ around $y$ and a $G$-invariant uniformity $U\in S$ so that $UV'\subset V$.
Let $U'\in S$ be a symmetric uniformity with $U'U'\subset U$.
By the assumptions there exists $\alpha_0$ such that for every $\alpha \geq \alpha_0$, 
$$
 g_\alpha x_\alpha\in V',~(x_\alpha,x) \in U'~\text{and}~(x'_\alpha,x)\in U'.
$$
Thus $(x_\alpha,x'_\alpha)\in U$ and, by the $G$-invariance of $U$, also $(g_\alpha x'_\alpha,g_\alpha x_\alpha)\in U$.
It follows that $g_\alpha x'_\alpha \in UV'\subset V$.

By switching the roles of $x_\ga$ and $x_\ga'$ we deduce that 
$\lim_T g_\alpha x_\alpha$ exists if and only if $\lim_T g_\alpha x'_\alpha$ exists, in which case they are equal.
The lemma follows by specializing to the constant net $x'_\alpha \equiv x$.
\end{proof}

\begin{lem}[A variant ot Mautner's lemma] \label{mautner}
Let $G$ be a topological group.
Let $X$ be a $G$-space equipped with a uniform structure $S$ and an $S$-compatible topology $T$.
Assume that the action is continuous with respect to both topologies $T$ and $T_S$ and 
equicontinuous with respect to $S$.
Let $g_\alpha$ be a net in $G$
and assume for some points $x,y\in X$, $y=T\text{-}{\lim} g_\alpha x$.
Assume $g\in G$ satisfies $\lim g^{g_\alpha^{-1}} =e$.
Then $gy=y$.
\end{lem}

\begin{proof}
By continuity of the action $G\act (X,T_S)$ we have
 $(T_S\text{-}\lim)  g^{g_\alpha^{-1}} x=x$.
Applying Lemma~\ref{CI} to the net $g_\alpha$ in $G$ and the net $g^{g_\alpha^{-1}} x$ in $X$,
we deduce that indeed
\[
 gy= g(T\text{-}{\lim}) g_\alpha x=(T\text{-}{\lim}) g g_\alpha x= (T\text{-}{\lim}) g_\alpha\cdot g^{g_\alpha^{-1}}x= (T\text{-}{\lim}) g_\alpha x= y.
\]
\end{proof}

\begin{lem} \label{inverseseq}
Let $G\act (X,S)$ be an equicontinuous action.
Assume that for some net $(g_\alpha)$ in $G$ and $x,y\in X$, $(T_S\text{-}{\lim})g_\alpha x= y$.
Then $(T_S\text{-}{\lim})g_\alpha^{-1}y= x$.
\end{lem}

\begin{proof}
For every neighborhood $V$ of $x$ there exists a $G$-invariant uniformity $U$ with $U(x)\subset V$.
By $g_\alpha x\to y$ there exists $\alpha_0$ such that for every $\alpha\geq \alpha_0$, $g_\alpha x\in U(y)$, that is $(g_\alpha x,y)\in U$.
By $G$-invariance we get $(x,g_\alpha^{-1}y)\in U$ and by symmetry $(g_\alpha^{-1}y,x)\in U$.
Therefore for every $\alpha\geq \alpha_0$, $g_\alpha^{-1}y\in U(x)\subset V$.
\end{proof}

\subsection{Universally closed maps and actions}\label{sec:UniClo}

Recall that a continuous map $\pi:X\to Y$ between topological spaces is called proper if the preimage of a compact set is compact, and closed if the image $\pi(A)$ of every closed set $A\subset X$ is closed in $Y$.
Under mild assumptions on $Y$, it is automatic that a proper map is closed.
This is the case if $Y$ is a K-space, e.g when $Y$ is either locally compact or satisfies the first axiom of countability, see \cite{Palais}.
In general however, a proper map is not necessarily closed.
The current section deals with the general case. Recall the following classical theorem:

\begin{thm}\label{thm:cmpt}
A topological space $K$ is compact if and only if for every topological space $Z$, the projection map $K\times Z\to Z$ is closed.
\end{thm}

Note that we do not assume any separation property of the topological spaces involved. Since we are not aware of a reference for \ref{thm:cmpt} in this generality, we add a proof for the convenience of the reader. 

\begin{proof}
The fact that if $K$ is compact then for every $Z$, $K\times Z\to Z$ is closed is standard and easy.
Assume now $K$ is not compact and pick a directed set $(\alpha)$ and a net $(x_\alpha)$ in $K$ which has no converging subnet.
For every $x\in K$ we can find a neighborhood $U_x$ and $\alpha_x$ such that for every $\alpha\geq \alpha_x$, $x_\alpha\notin U_x$.
Consider the poset obtained by adding to $(\alpha)$ a maximal element, $\infty$.
Observe that the collection consisting of all intervals in $(\ga)$ of the form $[\alpha,\infty]$ forms a base for a topology.
Let $Z$ be the topological space thus obtained.
Check that $\infty\in Z$ is not isolated.
Let $A\subset X\times Z$ be the complement of the open set $\cup_x (U_x\times [\alpha_x,\infty])$.
Observe that $A\cap X\times \{\infty\}=\emptyset$ and for each $\alpha$, $(x_\alpha,\alpha)\in A$, thus the projection of $A$ to $Z$ consists of the subset $Z-\{\infty\}$, which is not closed.
\end{proof}

Here is another basic result of point-set topology for which we couldn't find a proper reference.

\begin{thm} \label{CU}
Let $\pi:X\to Y$ be a continuous map between topological spaces.
The following are equivalent.
\begin{enumerate}
\item
For every topological space $Z$, the map $\pi\times \id{Z}:X\times Z\to Y\times Z$ is a closed map.
\item
$\pi$ is closed and proper.
\item
For every net $(x_\alpha)$ in $X$ which has no converging subnet,
the net $(\pi(x_\alpha))$ has no converging subnet in $Y$.
\end{enumerate}
\end{thm}

\begin{proof}
$(1)\Rightarrow (2):$
By taking $Z$ to be a point we see that $\pi$ is closed.
In order to see that $\pi$ is proper, consider an arbitrary compact subset $K\subset Y$ and an arbitrary topological space $Z$. The projection map $\pi^{-1}(K)\times Z\to Z$ is closed, being the composition of the closed maps $\pi^{-1}(K)\times Z \to K\times Z \to Z$. Thus by Theorem \ref{thm:cmpt} $\pi^{-1}(K)$ is compact. 

$(2)\Rightarrow (3):$
Assume by contradiction that $(x_\alpha)$ is a net in $X$ which has no converging subnet and $\pi(x_\alpha) \to y \in Y$.
Denote $X_y=\pi^{-1}(\{y\})$. Since $\pi$ is closed and proper, $X_y$ is non-empty and compact.
For every $x\in X_y$ we can find an open neighborhood $U_x$ of $x$ and $\alpha_x$ such that $\alpha \geq \alpha_x \Rightarrow x_\alpha\notin U_x$.
By compactness of $X_y$ we can find a finite set $F\subset X_y$ such that $X_y \subset \cup_{x\in F} U_x$.
We let 
$$
 V=Y\setminus \pi(X\setminus \cup_{x\in F} U_x).
$$
Since $\pi$ is closed $V$ is an open neighborhood of $y$ in $Y$. Note that $U=\pi^{-1}(V) \subset \cup_{x\in F} U_x$.
Let $\alpha_0$ be an index satisfying $\alpha_0\geq \alpha_x$ for every $x\in F$.
Then for every $\alpha\geq \alpha_0$, $x_\alpha\notin U$ and thus $\pi(x_\alpha)\notin V$, contradicting the assumption that $\pi(x_\alpha)\to y$.

$(3)\Rightarrow (1):$
Let $A\subset X\times Z$ be a closed set.
Assume, by way of contradiction, that $(\pi\times \id{Z})(A)$ is not closed in $Y\times Z$ and pick a net $(y_\alpha, z_\alpha)\in (\pi\times\id{Z})(A)$ converging to a point $(y, z) \notin (\pi\times\id{Z})(A)$.
Pick lifts $(x_\alpha)$ of $(y_\alpha)$ such that $(x_\alpha, z_\alpha) \in A$.
By our assumption, since $(y_\alpha)$ converges, $(x_\alpha)$ has a converging subnet.
Abusing the notation we assume that $(x_\alpha)\to x$.
It follows that $(x_\alpha, z_\alpha) \to (x, z)$.
Since $A$ is closed, $(x,z)\in A$ and thus $(y,z)=(\pi\times \id{Z})(x,z)\in (\pi\times\id{Z})(A)$. A contradiction.
\end{proof}

\begin{defn}
A map satisfying the above properties is called ``universally closed".
\end{defn}

Recall that a continuous action of $G$ on $X$ is called a {\em proper action} if the map 
\begin{equation} \label{actionmap}
 G\times X \to X\times X,~(g,x)\mapsto (x,gx)
\end{equation}
is a proper map.
Similarly, we say that the action is {\em universally closed} is the map (\ref{actionmap}) is universally closed.
Every universally closed action is proper.

\begin{prop} \label{UCaction}
If $G$ acts on $X$ and the action is universally closed then the point stabilizers are all compact and the quotient topology on the orbit space $X/G$ is Hausdorff. In particular, every orbit is closed.
\end{prop}

\begin{proof}
The fact that stabilizers are compact follows from the properness of the action.
To show that $X/G$ is Hausdorff, observe that
the set $X\times X\setminus \mathrm{Im}(G\times X)$ is open in $X\times X$ and hence its image
under the open map to $X/G\times X/G$ is open. Thus its complement, the
diagonal of $X/G\times X/G$, is closed.
\end{proof}

The following is a useful variant.

\begin{prop} \label{variant}
Suppose a topological group $G$ acts on $X$ and $T,T'$ are two topologies on $X$ such that the map 
\[ G\times (X,T) \to (X,T) \times (X,T'), \quad (g,x)\mapsto (x,gx) \]
is universally closed. Assume that points in $X$ are $T$-closed.
Then the stabilizers are compact and the $G$-orbits in $X$ are $T'$-closed.
\end{prop}

\begin{proof}
Again, compactness of the stabilizers follows from properness.
Given a point $x_0\in X$, the image of $G\times \{x_0\}$, that is $\{x_0\}\times Gx_0$, is a closed subset of $(X,T) \times (X,T')$ and its preimage in $X$
under the continuous map $(X,T')\to (X,T)\times (X,T')$, $x\mapsto (x_0,x)$ is the orbit $Gx_0$.
\end{proof}

\section{Quasi-semisimple groups} \label{qss}

The main objects of this paper are semisimple Lie groups over local fields. However, much of the things we prove are based on two specific properties, namely:
\begin{itemize}
\item the existence of a Cartan $KAK$ decomposition for $G$, 
\item for every $a\in A$, the group $G$ is generated by elements $g$ with the following property\footnote{In the classical case one can deduce this property using a root space decomposition of the Lie algebra.}: 
$\lim_{n\to\infty}a^nga^{-n}=1$, $\lim_{n\to-\infty}a^nga^{-n}=1$ or $\sup_{n\in\BZ}\| a^nga^{-n}\|<\infty$.
\end{itemize} 
This observation encourages us to introduce an axiomatic approach.
Indeed, formulating (variants of) the above as axioms will, on one hand, make our future arguments cleaner and more transparent, while on the other hand, our results will be more general, and apply for other classes of groups. 
Our axiomatic approach is influenced by \cite{ciob}.

Given a topological group $G$ and a net $g_\alpha$ in $G$ we define the following three groups:
\[ U^{(g_\alpha)}_+=\{x\in G~|~g_\alpha^{-1} x g_\alpha \to e\}, \quad
U^{(g_\alpha)}_-=\{x\in G~|~g_\alpha x g_\alpha^{-1} \to e\} \quad \mbox{and} \]
\[ U^{(g_\alpha)}_0=\{x\in G~|~\mbox{every subnet of both nets}~g_\alpha^{-1} x g_\alpha ~\mbox{and}~g_\alpha x g_\alpha^{-1}~\mbox{admits a converging subnet}\}. \]
The following lemma is obvious and left as an exercise to the reader.

\begin{lem} \label{U+-normal}
Let $G$ be a topological group and $g_\alpha$ a net in $G$.
The $U^{(g_\alpha)}_+$, $U^{(g_\alpha)}_-$ and $U^{(g_\alpha)}_0$ defined above are indeed groups and the group $U^{(g_\alpha)}_0$ normalizes both groups $U^{(g_\alpha)}_+$ and $U^{(g_\alpha)}_-$.
\end{lem}

\begin{defn}\label{def:qss}
A locally compact topological group $G$ is said to be quasi-semisimple (qss, for short) if there exists a closed subgroup $A<G$ satisfying the following axioms:
\begin{itemize}
\item There exists a compact subset $C\subset G$ such that $G=CAC$.
\item For every net $a_\alpha$ in $A$ with $a_\alpha \to \infty$, there exists a subnet $a_\beta$ such that 
the group $U^{(a_\beta)}_+$ is not pre-compact and the group generated by the three groups
$U^{(a_\beta)}_+$, $U^{(a_\beta)}_-$ and $U^{(a_\beta)}_0$ is dense in $G$.
\end{itemize}
\end{defn}

%

\begin{rem} \label{rem:qss}
The class QSS of quasi-semisimple groups is closed under finite direct products. 
Every compact group is qss and in addition if $H=G/O$ where $O\lhd G$ is a compact normal subgroup, then $H$ is qss iff $G$ is qss.   
\end{rem}

It is well known that Zariski connected semisimple groups over local fields are qss.
This follows for example from \cite[Ch. I, Proposition~(1.2.1)]{Margulis}. 
In particular every Zariski connected semisimple Lie group with finite center is qss.
It is not clear to us whether any quotient group of a general qss group is qss as well.
For a surjective map $\phi:G\to H$, where $G$ is qss relatively to a subgroup $A<G$, 
it is reasonable to expect that $H$ would be qss relatively to $\phi(A)$.
The only problem that might occur is that for some net $\phi(a_\beta) \to \infty$, the group $U_+^{\phi(a_\beta)}$ would be precompact in $H$.
This problem never occurs for semisimple groups.





\begin{thm} \label{thm:ss-qss}
Let $H$ be a locally compact topological group.
Let $k$ be a local field and $\mathbf{G}$ a Zariski connected, semi-simple algebraic group defined over $k$.
Assume there is a continuous surjection $\mathbf{G}(k)\twoheadrightarrow H$.
Then $H$ is qss.
In particular every Zariski connected semi-simple Lie group with finite center is qss.
\end{thm}

\begin{rem}
We will see in Corollary~\ref{cor:imagess} that the assumption that $H$ is locally compact is in fact redundant.
\end{rem}

\begin{proof}
Since $\mathbf{G}(k)$ is locally compact, by a standard Baire category argument, $H$ is isomorphic as a topological group to $\mathbf{G}(k)/N$
for some closed normal subgroup $N\lhd \mathbf{G}(k)$.
We denote by $\phi:\mathbf{G}(k) \to \mathbf{G}(k)/N$ the natural surjection.
Recall that $\mathbf{G}$ contains a closed normal subgroup $\mathbf{G}(k)^+$ such that  $\mathbf{G}(k)/\mathbf{G}(k)^+$ is abelian and compact, and even finite when $\mathrm{char}(k)=0$.
The group $\mathbf{G}(k)$ is the image of the natural isogeny $\tilde{\mathbf{G}}(k) \to \mathbf{G}(k)$ where $\tilde{\mathbf{G}}$ is the simply connected form of $\mathbf{G}$.
By \cite[Ch. I, Proposition~1.5.4(vi)]{Margulis}, $\mathbf{G}(k)=\mathbf{G}(k)^+\cdot Z_{\mathbf{G}}(\mathbf{S})(k)$ where $\mathbf{S}$ is a $k$-split torus.
We let $A=\phi(\mathbf{S}(k))$.
In view of the discussion above, our only task is to show that for a net $(s_\alpha)$ in $\mathbf{S}(k)$ which tends to $\infty$ mod $N$, one can pass to a subnet $(s_\beta)$ such that $U^{\phi(s_\beta)}$ is not precompact.
We abuse notation and view $\mathbf{S}$ a subgroup of $\tilde{G}$.
Consider the preimage of $N$ in $\tilde{\mathbf{G}}(k)$ and mod out the finite center.
This is a normal subgroup in a product of simple non-abelian groups, thus consists of a product of factors.
Moding out these factors, we still have that $s_\alpha\to\infty$.
The non-pre-compactness of $U^{\phi(s_\beta)}$ follows by a standard root space decomposition argument.
\end{proof}

\begin{rem}[Adelic groups are qss]
Let $\BK$ be a global field and $\mathbf{G}$ a Zariski connected, simply connected, semisimple $\BK$-algebraic group. 
Assume that $\mathbf{G}$ has no anisotropic factor. Let $\BA=\BA_\BK$ be the associated ring of adels. Then $\mathbf{G}(\BA)$  is qss. To see this recall that $\mathbf{G}(\BA)$ is the restricted topological product of $\mathbf{G}(\BK_v)$ relative to the open compact subgroups $\mathbf{G}(\OO_v)<\mathbf{G}(\BK_v)$ where $v$ runs over the finite valuations, and $\OO_v$ is the local ring of $\BK_v$. The reason $\mathbf{G}(\BA)$ is qss is that the same subgroups $\mathbf{G}(\OO_v)$ used in the construction of restricted topological product can be used in the associated $CAC$ (or rather $KAK$) decomposition of the corresponding factors $\mathbf{G}(\BK_v)$. It is easy to verify the details.
\end{rem}

Another family of qss groups is given by the following (see \cite{ciob} and \cite{caprace-ciob}):

\begin{thm} \label{thm3.6}
Let $G$ be a group acting strongly transitively on an affine building.
Then $G$ is qss.
In particular every group of automorphisms of a simplicial tree whose action on the boundary of the tree is 2-transitive is qss.
\end{thm}

We note that the first ones to implicitly use the qss axioms for a boundary 2-transitive tree group are Burger and Mozes, in their proof of the Howe--Moore property for such groups in \cite{BM}.

\section{The main theorem} \label{sec:main}


The main result of this paper is the following general statement:

\begin{thm} \label{HM}\label{mainthm}
Let $G$ be a quasi-semisimple group.
Let $X$ be a $G$-space equipped with a uniform structure $S$ and an $S$-compatible topology $T$.
Assume that the action is continuous with respect to both topologies $T$ and $T_S$ and 
equicontinuous with respect to $S$.
Suppose that no non-compact normal subgroup of $G$ admits a global fixed point in $X$.
Then the map
\[ \phi:G\times (X,T_S) \to (X,T_S) \times (X,T), \quad (g,x)\mapsto (x,gx) \]
is universally closed. In particular, it is proper.
\end{thm}

Applying Proposition~\ref{variant} we get the following.

\begin{cor} \label{orbits}
Under the conditions of Theorem~\ref{mainthm} we have
\begin{enumerate}
\item the stabiliser in $G$ of every point in $X$ is compact, and 
\item the $G$-orbits in $X$ are $T$-closed.
\end{enumerate}
\end{cor}

In the special case where $T=T_S$ we obtain Theorem~\ref{demo} presented in the introduction.
Moreover, we get:

\begin{cor}\label{ref:Hausdorff}
With respect to the quotient topology induced from $T_S$, the orbit space $X/G$ is Hausdorff and completely regular.
\end{cor}

By Lemma~\ref{uniformquotient}, $X/G$ admits a uniform structure, hence it is Hausdorff and completely regular given that it is $T_0$,
but it is $T_1$ by the above discussion. To see directly the Hausdorff property of $X/G$,
consider two points $x,y$ which do not belong to the same orbit. Since $Gy$ is closed, we have an open neighborhood $V$ of $x$ which is disjoint from $Gy$. Consider a $G$-invariant uniformity $U$ such that $UU(x)\subset V$ and pick a symmetric uniformity $U'$ contained in $U$. It is easy to verify that the open sets $GU'(x)$ and $GU'(y)$ are disjoint.


%
\begin{proof}[Proof of Theorem \ref{mainthm}]
By way of contradiction we assume that the map $\phi$ is not universally closed and show eventually the existence of a point fixed by some non-compact normal subgroup of $G$. The proof consists of four steps.  


Throughout the proof we let $A<G$ be the subgroup guaranteed by the qss assumption, and let $C$ be a compact subset of $G$ such that $G=CAC$.

\medskip

\noindent
{\em Step 1: There exist points $x,y \in X$ and a net $a_\alpha\in A$ satisfying $a_\alpha\to \infty$ and $(T\text{-}\lim) a_\alpha x=y$.}

\medskip

In view of Theorem \ref{CU}, the assumption that $\phi$ is not universally closed is equivalent to the existence of a directed set $(\alpha)$ and a net $(g_\alpha,x_\alpha)$ which has no converging subnet, such that the net
$(x_\alpha,g_\alpha x_\alpha)$ converges in the $T_S\times T$-topology.

Let $g_\alpha=c_\alpha a_\alpha c'_\alpha$ be a corresponding $CAC$ expression of the elements $g_\alpha$.
Upon passing to a subnet we may assume that both $c_\alpha$ and $c'_\alpha$ converge in $C$. Note that necessarily $a_\alpha$ has no converging subnet in $A$, that is $a_\alpha \to \infty$.

Denote 
$$
 c=\lim c_\alpha~\text{and}~c'=\lim c'_\alpha,
$$
and set
$$
 x=c' (S\text{-}\lim) x_\alpha~\text{and}~y=c^{-1} (T\text{-}\lim) g_\alpha x_\alpha.
$$
Since $G$ acts continuously on $(X,T_S)$, we have
\[ 
 (T_S\text{-}{\lim}) c'_\alpha x_\alpha = x.
\]
Since $G$ acts continuously on $(X,T)$, we have
\[ 
 (T\text{-}{\lim}) a_\alpha c'_\alpha x_\alpha=(T\text{-}{\lim})  c^{-1}_\alpha \cdot g_\alpha x_\alpha=
\lim c_\alpha^{-1} \cdot (T\text{-}{\lim}) g_\alpha x_\alpha = c^{-1} (T\text{-}{\lim}) g_\alpha x_\alpha =y. \]
Applying Lemma~\ref{CI} to the net $a_\alpha$ in $G$ and the net $c'_\alpha x_\alpha$ which $T_S$-converges to $x$ in $X$,
we deduce that $y= (T\text{-}\lim) a_\alpha x$.

\medskip
\noindent
{\em Step 2 (reducing to the case $T=T_S$): The action of $G$ on $(X,T_S)$ is not universally closed.}

\medskip

By Step 1, and by the second property in Definition \ref{def:qss}, replacing the net $(a_\ga)$ by a subnet $(a_\gb)$, we have in addition to 
\begin{itemize}
\item $a_\gb\to \infty$ and
\item $(T\text{-}\lim) a_\gb x=y$,
\end{itemize}
that 
\begin{itemize}
\item $U^{(a_\gb)}_+$ is not precompact.
\end{itemize}
For $g\in U^{(a_\beta)}_+$ we have 
$\lim g^{a_\beta^{-1}} =1$, hence by Lemma~\ref{mautner}, $gy=y$.
Thus the stabilizer of $y$ is non-compact.
By Proposition~\ref{UCaction} it follows that the action of $G$ on $(X,T_S)$ is not universally closed.


\medskip
\noindent
{\em Step 3: There exist a point $x \in X$ and a net $a_{\beta'}\in A$ satisfying $a_{\beta'}\to \infty$ and $(T_S\text{-}\lim) a_{\beta'} x=x$.}

\medskip

By Step 2 we know that the map
\[ G\times (X,T_S) \to (X,T_S) \times (X,T_S), \quad (g,x)\mapsto (x,gx) \]
is not universally closed.
We thus may apply Step 1 in the special case $T=T_S$ and obtain points $x,y \in X$ and a net $a_\alpha\in A$ satisfying $a_\alpha\to \infty$ and 
$(T_S\text{-}\lim) a_\alpha x=y$.
By Lemma~\ref{difference},
there exists a directed set $(\beta')$ and two nets $n,n':(\beta')\to A$ satisfying
$n(\beta')x\to y$ and $n'(\beta')x \to y$ in $X$ (all limits in $X$ here are with respect to $T_S$) and $n(\beta')^{-1}n'(\beta') \to \infty$ in $A$.
By Lemma~\ref{inverseseq}, $n(\beta')^{-1} y \to x$.
Applying Lemma~\ref{CI} (in the special case $T=T_S$)
with respect to the directed set $(\beta')$, the net $n'(\beta')x$ in $X$
and the net $n(\beta')^{-1}$ in $A$, 
we conclude that $n(\beta')^{-1}n'(\beta')x \to x$.
We are done by setting $a_{\beta'}=n(\beta')^{-1}n'(\beta')$.

\medskip
\noindent
{\em Step 4: There exists a point in $X$ which is fixed by a non-compact normal subgroup of $G$.}

\medskip

We let $x$ be a point as obtained in Step 3. We will show that its stabilizer $G_x$ contains a normal non-compact subgroup of $G$.
By replacing the net obtained in Step 3 by a subnet, using the qss second axiom we get a net $(a_{\ga'})$ in $A$ satisfying the following properties:
\begin{itemize}
\item  $a_{\ga'}\to \infty$.
\item $(T_S\text{-}\lim) a_{\ga'} x=x$.
\item $U^{(a_{\ga'})}_+$ is not precompact.
\item The group generated by the three groups
$U^{(a_{\ga'})}_+$, $U^{(a_{\ga'})}_-$ and $U^{(a_{\ga'})}_0$ is dense in $G$.
\end{itemize}

In view of Lemma~\ref{mautner}, $U^{(a_{\ga'})}_+<G_x$.
Moreover, by Lemma~\ref{inverseseq} we also have 
$$
 (T_S\text{-}\lim) a_{\ga'}^{-1} x=x,
$$ 
which by Lemma~\ref{mautner} gives $U^{(a_{\ga'})}_-<G_x$.
By Lemma~\ref{U+-normal}, the closed group generated by the subgroups $U^{(a_{\ga'})}_+$ and $U^{(a_{\ga'})}_-$ is normal in $G$.
It is non-compact as $U^{(a_{\ga'})}_+$ is not precompact.
We conclude that $G_x$ contains a normal non-compact subgroup of $G$, completing the argument by contradiction.
\end{proof}


\section{Images of homomorphisms}\label{sec:topologies}\label{sec:closed-image}

The main result of this section is the following theorem, whose proof will be completed at the end of the section.

\begin{thm}\label{thm:image}
Let $G$ be a qss group and $H$ an arbitrary Hausdorff topological group.
Let $\phi:G\to H$ be a continuous homomorphism. Then $\phi(G)$ is closed.
If further $G$ is separable then the induced map $G/\ker(\phi) \to \phi(G)$ is a homeomorphism and in particular $\phi(G)$ is locally compact.
\end{thm}

We know by Theorem~\ref{thm:ss-qss} that every locally compact image of a semisimple group is qss. 
Applying the above theorem we obtain:

\begin{cor}\label{cor:imagess}
Let $G$ be a semisimple analytic group with a finite center (the $k$ points of a Zariski connected semisimple algebraic  group $\mathbf{G}$, defined over a local field $k$).
Let $H$ be a Hausdorff topological group.
Let $\phi:G\to H$ be a continuous homomorphism. Then $\phi(G)$ is closed in $H$ and the induced map $G/\ker(\phi) \to \phi(G)$ is a homeomorphism.
In particular, $\phi(G)$ is locally compact and it is also qss.
\end{cor}

Another corollary of Theorem~\ref{thm:image} regards minimality of group topologies.
Given a group $G$, a group topology on $G$ is a topology with respect to which $G$ becomes a topological group.
We say that a Hausdorff topological group $G$ is topologically-minimal if there are no weaker Hausdorff group topologies on $G$. 

\begin{cor}\label{weakest-top}
Every factor group of a separable qss group is topologically-minimal.
\end{cor}

\begin{proof}
Let $G$ be a separable qss group and let $N\lhd G$ a closed normal subgroup. Denote the quotient topology on $G/N$ by $T$. 
Let $T'\subset T$ be a Hausdorff group topology on $G/N$.
By setting $H=(G/N,T')$ and applying Theorem~\ref{thm:image} we conclude that $T'=T$.
\end{proof}

\subsection{Closed images}

We first prove the first part of Theorem~\ref{thm:image}.
 
\begin{prop}\label{prop:image}
Let $G$ be a qss group, $H$ an arbitrary Hausdorff topological group and $\phi:G\to H$ a continuous injective homomorphism. Then $\phi(G)$ is closed in $H$.
\end{prop}

\begin{proof}
Set $X=H$ and consider the left $G$ action on $X$.
Endow $H$ with the left uniform structure described in Example~\ref{uniform-group}.
This uniform structure is invariant for the left regular action of $H$, and in particular under the $G$ action,
thus the assumptions of Theorem \ref{mainthm} hold.
By Corollary~\ref{orbits} the $G$-orbits are closed. Since the image of $\phi$ coincides with the orbit of the identity $1_H$, the proposition is proved.
\end{proof}

Note that a similar theorem was proven by Omori~\cite{omori} for a class of connected Lie groups, including all connected semisimple Lie groups with finite center, under the assumption that the target group $H$ satisfies the first axiom of countability.

\subsection{Group topologies}

In this subsection we set some preliminaries regarding group topologies.
If $T$ is a group topology on $G$, setting 
$T(e)=\{U\mid e\in U \in T\}$
where  $e\in G$ denotes the identity element,
it is standard that $T=\{gU\mid g\in G,~U\in T(e)\}$ and
\begin{enumerate}
\item For all $U\in T(e)$ there exists $V\in T(e)$ such that $V\cdot V\subset U$.
\item For all $U\in T(e)$ and $g\in G$ also $U^{-1}, U^g\in T(e)$.
\end{enumerate}
The following lemma is straightforward and we leave its verification to the reader.

\begin{lem}\label{lem:5.5}
Let $T$ be a group topology on $G$ and $A\subset G$ a dense subset.
Assume $C\subset T(e)$ is a collection satisfying
\begin{enumerate}
\item For all $U\in C$ there exists $V\in C$ such that $V\cdot V\subset U$.
\item For all $U\in C$ and $g\in A$ also $U^{-1}, U^g\in C$.
\end{enumerate}
Then the topology generated by the collection 
$\{gU\mid g\in A,~U\in C\}$ is a group topology on $G$ (included in $T$).
\end{lem}

A topology is said to be countably generated (or second countable) if it is generated as a topology by a countable sub-collection.
The following proposition is a useful step in the proof of the second part of Theorem ~\ref{thm:image}.

\begin{prop} \label{prop:separable}
Let $G$ be a group and $T$ be a separable group topology on $G$.
Let $D\subset T$ be a countable sub-collection. Then there exists a group topology $T'\subset T$ which is countably generated and such that $D\subset T'$.
\end{prop}

\begin{proof}
Choose a countable dense subset $A\subset G$.
For each element $U\subset D$ choose $g_U\in A\cap U$ and set $C_1=\{g_U^{-1}U\mid U\in D\}$.
Clearly $C_1\subset T(e)$ is countable and $D\subset \{gU\mid g\in A,~U\in C_1\}$.
We construct countable collections $C_n\subset T(e)$ inductively as follows: 
given $C_n$, for each $U\in C_n$ we choose $V_U\in T(e)$ such that $V_U\cdot V_U\subset U$ and set
\[ C_{n+1}=\{V_U\mid U\in C_n\}\cup \{U^{-1}\mid U\in C_n\} \cup \{U^g\mid U\in C_n,~g\in A\}.\]
Setting $C=\cup C_n$, it follows from Lemma \ref{lem:5.5} that the topology generated by the countable collection $\{gU\mid g\in A,~U\in C\}$ is a group topology on $G$,
which contains $D$.
\end{proof}

\subsection{Completions}

Given a uniform space, as in the case of a metric space, one can consider Cauchy nets, that is nets $(x_\alpha)$ such that for every uniformity $U$ there exists some $\alpha_U$ such that for every $\alpha,\alpha' \geq \alpha_U$, $(x_\alpha,x_\alpha')\in U$.
The uniform space is said to be complete if every Cauchy net converges.
Any uniform space could be completed
by adding ideal points: equivalence classes of Cauchy nets. Two Cauchy nets $(x_\ga)$ and $(y_\gb)$ are equivalent if for any uniformity $U$ there are $\ga_U$ and $\gb_U$ such that for every $\ga\ge\ga_U$ and $\gb\ge\gb_U$, $(x_\ga,y_\gb)\in U$. 

Fix a topological group $G$ and recall the definition of the left uniform structure given in  Example~\ref{uniform-group}.
If it is complete than $G$ is said to be Weil-complete (or left-complete).
Other uniform structures of interest on a topological group are the right uniform structure, defined similarly to the left one, and the two sided uniform structure, which is the finest uniform structure contained in both the left and the right ones.
A group is said to be Raikov-complete if its two sided uniform structure is complete.
It is known (see e.g \cite[Chapter~3.6]{ArTk}) that a Weil-complete group is also Raikov-complete (and right-complete).

\begin{prop} \label{prop:weil}
Let $G$ be a qss group. Let $T$ be group topology on $G$ included in the original topology of $G$.
Then $(G,T)$ is Weil-complete. In particular it is Raikov-complete.
\end{prop}

\begin{proof}
Set $X$ to be the completion of $G$ with respect to the left uniform structure associated with the group topology $T$ and note that the left action of $G$ on itself extends to an action of $G$ on $X$ which satisfies the assumptions of Theorem \ref{mainthm}. Thus $G\subset X$ is a closed $G$-orbit. Since $G$ is dense in $X$, we conclude that $X=G$, thus $G$ is Weil-complete.
As remarked above, it follows that $G$ is also Raikov-complete.
An alternative way to see that $G$ is Raikov-complete would be by repeating the argument taking $X$ to be the completion of $G$ with respect to the two sided uniform structure.
\end{proof}

\subsection{Baire property}

Recall that a topological group is said to be a Baire group if its underlying topological space is a Baire space, that is the Baire Category Theorem holds true.

A uniform structure is said to be countably generated if it contains a countable collection of uniformities which is not contained in any proper sub-uniform-structure. 
It is a standard fact that a countably generated uniform structure is given by an equivalent pseudo metric on the space, which is a true metric iff the generating uniformities separate points. 
The proof of the following Lemma is an obvious adjastment of the standard proof that a complete metric space is Baire.
Alternatively, one can reduce it to the metric case by introducing an equivalent metric.

\begin{lem} \label{lem:baire}
The underlying topological space of a complete, countably generated uniform space is Baire.
\end{lem}

Noticing that the left and right uniform structures of a topological group are countably generated iff the topology is first countable (the identity has a countable basis of neighborhoods)
we get:

\begin{cor} \label{cor:baire}
A Raikov-complete first countable group is Baire. In particular, a Weil-complete first countable is Baire.
\end{cor}

The final step in the proof of Theorem~\ref{thm:image} is given by:

\begin{prop}\label{prop:baire}
Let $G$ be a separable qss group. Let $T$ be a group topology on $G$ included in the original topology of $G$.
Then $(G,T)$ is a Baire group.
\end{prop}

\begin{proof}
Let $U_n$ be a countable collection of open dense sets in $T$ and denote by $V$ the complement of the $T$-closure of their intersection.
Set $D=\{V\}\cup\{U_n\mid n\}$ and use Proposition~\ref{prop:separable} to find a first countable group topology $T'$ on $G$ such that $D\subset T'\subset T$.
Note that $(G,T')$ is Weil-complet by Proposition~\ref{prop:weil}. We deduce that $V\ne G$ by the fact that $(G,T')$ is a Baire group which follows from Corollary~\ref{cor:baire}.
\end{proof}

The above result is related to \cite[Corollary~3.3]{DU98}.

\subsection{Proof of {Theorem~\ref{thm:image}}}

In view of Proposition~\ref{prop:image} we only need to verify the last sentence,
namely that, assuming $G$ is separable, the induced map $G/\ker(\phi) \to \phi(G)$ is a homeomorphism.
Note that $G$, being seprable, is $\sigma$-compact.
Note further that $\phi(G)$ has the Baire property, by Proposition~\ref{prop:baire}.
In such a case it is well known that $\phi:G\to \phi(G)$ is open, and the proof follows.

\section{Measurable metrics and metric ergodicity} \label{sec:ME}

Theorem~\ref{mainthm} and Corollary~\ref{orbits} (1) could be applied in the case where $X$ is a metric space, taking the metric uniform structure and $T=T_S$.
We obtain:

\begin{thm} \label{thm:iso}
Let $G$ be a quasi-semisimple group.
Assume that $G$ acts isometrically and continuously on a metric space $X$
and suppose that no non-compact normal subgroup of $G$ admits a global fixed point in $X$.
Then the $G$-orbits are closed in $X$ and the points stabilizers are compact in $G$.
\end{thm}

The following theorem has many ergodic theoretical applications.

\begin{thm} \label{ME} \label{thm6.2}
Let $G$ be a Zariski connected, semisimple analytic group with a finite center (the $k$ point of a Zariski connected semisimple algebraic  group $\mathbf{G}$, defined over a local field $k$).
Let $H<G$ be a closed subgroup.
Suppose that $G/H$ admits a $G$-invariant, separable, measurable metric.
Then $H$ contains a factor of $G$ as a cocompact subgroup.
\end{thm}

In case the metric is continuous, this theorem is an immediate application of Theorem~\ref{thm:iso}.
Indeed, the associated uniform structure on $G/H$ is $G$-invariant and continuous.
Replacing $G$ by $G/N$ where $N$ is the action kernel, using the fact that $G/N$ is qss (Theorem~\ref{thm:ss-qss}) we see that $H$, being the stabiliser of a point, must be compact.
The fact that the theorem applies also for measurable metrics is a consequence of the following:



\begin{lem} \label{measd}
Let $G$ be a locally compact group and $H<G$ a closed subgroup.
Denote by $T$ the standard topology on $G/H$.
Let $d$ be a $G$-invariant, separable, measurable metric on $G/H$.
Then $d$ is $T$-continuous.
\end{lem}

\begin{proof}
We will prove that $T_d\subset T$.
Let $\pi:G\to G/H$ be the quotient map. 
By the definition of the topology $T$ on $G/H$, $\pi$ is $T$-open, so it is enough to show that $\pi$ is $T_d$-continuous.
By $G$-invariance it is enough to show continuity at $e$.
Denote by $B(\epsilon)$ the $d$-ball of radius $\epsilon$ centered at $\pi(e)$.
We need to find for every $\epsilon>0$ an identity neighborhood $U$ in $G$ whose image is in $B(\epsilon)$.
For a given $\epsilon>0$ fix a countable cover of $G/H$ by balls of radius $\epsilon/2$.
At least one of the preimages of the balls is not Haar null, hence also the set $A=\pi^{-1}(B(\epsilon/2))$ is not null.
One easily checks that $A=A^{-1}$ and $\pi(AA)\subset B(\epsilon)$.
Moreover, it is well known that $AA^{-1}$ contains an identity neighborhood $U$, as desired.
\end{proof}


\begin{thm} \label{thm6.4}
Let $G$ be a semisimple analytic group with a finite center (the $k$ point of a Zariski connected semisimple algebraic  group $\mathbf{G}$, defined over a local field $k$).
Let $H<G$ be a closed subgroup.
Assume there exists a metric $d$ on $G$ which is separable, measurable, left $G$-invariant and right $H$-invariant.
Then $H$ is compact.
\end{thm}


\begin{proof}
By Lemma~\ref{measd}, $d$ is continuous.
By Theorem~\ref{ME}, $H$ contains cocompactly a factor $G_1$ of $G$. 
Let $X$ be $G_1$ endowed with the induced metric.
$X$ is a $G_1\times G_1$-space for the left and right actions which preserve the metric.
It follows by Corollary \ref{orbits} (1) that the stabilizer of $e$, namely the diagonal copy of $G_1$, is compact.
It follows that $H$ is compact as well. 
\end{proof}

\begin{defn}
Let $G$ be a group. Let $X$ be a $G$-Lebesgue space, that is a standard Borel space endowed with a measure class, on which $G$ acts measurably, preserving the measure class.
The action of $G$ on $X$ is said to be {\em metrically ergodic} if for every separable metric space $U$ on which $G$ acts isometrically,
every $G$-equivariant measurable function from $X$ to $U$ is a.e. a constant.
\end{defn}

\begin{thm} \label{ME-struct} \label{thm6.6}
Let $G$ be a semisimple analytic group with a finite center (the $k$ point of a Zariski connected semisimple algebraic  group $\mathbf{G}$, defined over a local field $k$).
Let $H<G$ be a closed subgroup. 
Endow $G/H$ with the unique $G$-invariant Radon measure class.
Then $G/H$ is $G$-metrically ergodic if and only if 
the image of $H/G_1$ is not precompact in $G/G_1$ for every proper factor group $G_1\lhd G$.

An ergodic $G$-Lebesgue space $X$ is not metrically ergodic if and only if it is induced from an ergodic $H$-space,
for some closed subgroup $H<G$ which contains cocompactly a factor group $G_1\lhd G$ with $G/G_1$ non-compact.
\end{thm}

\begin{proof}
Let $G_1\lhd G$ be a proper normal subgroup and suppose that $H'=\overline{HG_1}/G_1$ is compact in $G'=G/G_1$.
Pick a positive function $f\in L^2(G')$ and average it over the right action by $H'$, using the Haar measure on $H'$.
The function obtained is $H'$-invariant, but not $G'$-invariant (as $G'$ is non-compact), thus provides a non-constant $G'$-equivariant map $G'/H'\to L^2(G')$.
Precomposing with the map $G/H\to G'/H'$ we disprove the metric ergodicity of $G/H$.

In addition, given a $G$-space $X$ of the form $X=\text{Ind}_H^G(X')$ where $X'$ is an $H$-space on which $H$ acts with co-compact kernel. Since $H$ contains the unimodular group $G_1$ as a cocompact subgroup, it must be unimodular as well, and the procedure above produces a non-constant $G$-map from $X$ to $L^2(G/H)$.

Let now $X$ be an ergodic $G$-Lebesgue space which is not metrically ergodic,
and let $\phi:X\to U$ be a $G$-equivariant map to a separable metric $G$-space.
Let $G_1$ be the maximal factor of $G$ for which the image of $X$ is essentially contained in $U^{G_1}$
and let $G'=G/G_1$.
By ergodicity of $X$ we assume as we may that $\phi(X)$ intersects nully the fixed points set of all proper factors of $G'$ in $U^{G_1}$.
Replacing $U$ with $U^{G_1}$ minus the union of these fixed points sets, we may assume that the action of $G$ on $U$ factors through $G'$
and that proper factors of $G'$ have no fixed points. 
By Corollary~\ref{ref:Hausdorff} $U/G'$ is Hausdorff. Hence by the ergodicity of $X$, $\phi(X)$ is essentially supported on a unique orbit, which we identify
with $G'/H'$ for some closed subgroup $H'<G'$. By Corollary~\ref{orbits}, $H'$ is compact in $G'$.
Letting $H$ be the preimage of $H'$ in $G$, we deduce that $X$ is induced from $H$.

In particular, it follows that if $X=G/H$ is $G$-metrically ergodic then
the image of $H$ is not precompact in $G/G_1$ for every proper factor group $G_1\lhd G$.
\end{proof}

The fact that metric ergodicity is preserved by a restriction to a lattice is general. 
We record it here for reference.

\begin{cor} \label{cor6.7}
Let $G$ be a semisimple analytic group with a finite center, and $\Gamma$ a lattice in $G$.
Then every metrically ergodic $G$-space $Y$ is also $\gC$-metrically ergodic.

In particular $\gC$ acts metrically ergodically on $G/H$ whenever $H\le G$ is a closed subgroup whose image in every proper quotient of $G$ is not pre-compact.
\end{cor}

\begin{proof}
Assume that $\phi:Y\to U$ is a $\Gamma$-equivariant measurable map into a separable metric space on which $\Gamma$ acts isometrically.
Replacing if necessary the metric $d$ on $U$ by $\min\{d,1\}$ we assume that $d$ is bounded.
Consider the space of $\Gamma$-equivariant measurable maps, defined up to null sets, $L(G,U)^\Gamma$, endowed with the metric 
\[ D(\alpha,\beta)=\sqrt{\int_{\Gamma \backslash\! G} d(\alpha(x),\beta(x))^2dx} \]
where the integration is taken over a fundamental domain for $\Gamma$ in $G$.
Define the map $\psi:Y\to L(G,U)^\Gamma$ by $\psi(y)(g)=\phi(gy)$.
Note that indeed, $\psi(y)$ is $\Gamma$-invariant, and further $\psi$ intertwines the $G$-action on $Y$ and the $G$-action on $L(G,U)^\Gamma$
coming from the right regular action of $G$.
By $G$-metric ergodicity of $Y$ we conclude that $\psi$ is essentially constant.
The essential image is a $G$-invariant function on $G$, thus a constant function to $U$.
This constant in turn is the essential image of $\phi$, thus $\phi$ is essentially constant as well.
\end{proof}

Recall that for probability measure preserving actions, metric ergodicity is equivalent to the weak mixing property.

\begin{cor} \label{statME} \label{cor6.8}
Let $G$ be a semisimple analytic group with a finite center and no compact factors.
Let $\mu$ be an admissible probability measure on $G$.
Let $(X,\nu)$ be a  $G$-Lebesgue space endowed with a $\mu$-stationary ergodic probability measure.
Then $X$ is metrically ergodic.
In particular, if the action on $X$ is measure preserving then $X$ weakly mixing (and in fact it is mixing modulo the action kernel).
\end{cor}

In fact, in the measure preserving case, $G'\acts X$ is even mixing, as we shall see in \ref{HM}.
Below we sketch the proof of the corollary.
Since we do not want to dive into the details of the subject here, we address the interested reader to \cite{stat-struct} for further details and clarifications.
Assume that $X$ is not metrically ergodic.
By Theorem~\ref{ME-struct}, there exists a (non-compact) quotient group $G'$, a compact group $H'<G'$ and an equivariant map $\phi:X\to G'/H'$.
Denote $\nu'=\phi_*(\nu)$.
Since $\nu'$ is recurrent with respect to a random sequence in $G$, while the action is dissipative, we get a contradiction.
We further remark that by the theory of Furstenberg--Poisson boundaries, it is a general fact that the question of metric ergodicity of a stationary measure reduces to the invariant measure case.
Indeed, the Furstenberg--Poisson boundary of a group, with respect to an admissible measure, is always a metrically ergodic action.
It follows that for a stationary space $X$ and an equivariant map into a metric space, $X\to U$, the pushed measure is invariant:
the associated boundary map from the Furstenberg--Poisson boundary to $\Prob(U)$ must be constant, due to the existence of a natural invariant metric on $\Prob(U)$. 
Thus the corollary above is reduced to the classical theorem of Howe--Moore, Theorem~\ref{HM} which we will prove independently.

\begin{cor} \label{cor6.9}
Let $G$ be a semisimple analytic group with a finite center.
Let $Y$ be a metrically ergodic $G$-space.
Let $X$ be an ergodic probability measure preserving $G$-Lebesgue space.
Then the diagonal action of $G$ on $X\times Y$ is metrically ergodic.
\end{cor}

\begin{proof}
Assume that $\phi:X\times Y\to U$ is a $G$-equivariant measurable map into a separable metric space on which $G$ acts isometrically.
By replacing if necessary the metric $d$ on $U$ by $\min\{d,1\}$ we may assume that $d$ is bounded.
Consider the space of measurable maps, defined up to null sets, $L(X,U)$, endowed with the metric 
\[ 
 D(\alpha,\beta)=\sqrt{\int_{X} d(\alpha(x),\beta(x))^2dx}. 
\]
Define the map $\psi:Y\to L(X,U)$ by $\psi(y)(x)=\phi(x,y)$.
Note that $\psi$ is $G$-equivariant.
By the $G$-metric ergodicity of $Y$ we conclude that $\psi$ is essentially constant.
The essential image is a $G$-equivariant map from $X$ to $U$.
By Corollary~\ref{statME}, $X$ is metrically ergodic as well, thus the latter map is also essentially constant.
It follows that $\phi$ was essentially constant to begin with.
\end{proof}

\section{Monoid compactifications}\label{sec:monoid}

\subsection{Ellis joint continuity}

Let $G$ be a topological group, $X$ a topological space and $G\times X\to X$ an action.
We will say that the action is {\em separately continuous} if for every $x_0\in X$ and $g_0\in G$ both maps
\[ G\to X,~g\mapsto gx_0 \quad\mbox{and}\quad X\to X,~x\mapsto g_0x \]
are continuous. We will say that the action is {\it jointly continuous} if the map
$$
 G\times X\to X,~(g,x)\mapsto gx
$$
is continuous.

\begin{lem} \label{strong->join}
Let $G$ be a topological group, $X$ a locally compact topological space and $G\times X\to X$ a separately continuous action.
Consider the left regular action of $G$ on $C_0(X)$ endowed with the sup-norm topology. Then the following are equivalent:
\begin{enumerate}
\item
The action of $G$ on $X$ is jointly continuous.
\item
For every $f\in C_0(X)$, the orbit map $G\to C_0(X)$ given by $g\mapsto f(g^{-1}\cdot)$ is continuous.
\item
The action of $G$ on $C_0(X)$ is jointly continuous.
\end{enumerate}
\end{lem}

\begin{proof}
The fact that (1) implies (3) is standard.
Clearly (3) implies (2) (in fact, the converse implication is given by Lemma~\ref{Joint}).
 We prove that (2) implies (1).
By Urysohn's lemma, the collection of subsets of $X$ of the form $f^{-1}(W)$ for $f\in C_0(X)$ and $W$ open in $\bbC$ is a subbasis for the topology.
Fixing $f$ and $W$, our aim is to show that for every $g\in G$ and $x\in X$ with $gx\in f^{-1}(W)$ there exists an open set $(g,x)\in U\times V\subset G\times X$ such that $U\cdot V\subset f^{-1}(W)$.
Choose $\epsilon>0$ for which the disc $B(f(gx),\epsilon)\subset W$ and let
$$
 V=(g^{-1}f)^{-1}(B(g^{-1}f(x),\epsilon/2)).
$$
Let $U^{-1}\subset G$ be the preimage of $B(g^{-1}f,\epsilon/2)\subset C_0(X)$ under the $f$-orbit map $G\to C_0(X),~h\mapsto h^{-1}f$.
Then $U$ is open by our continuity assumption, and
for $h\in U, y\in V$,
\[ 
 |f(hy)-f(gx)| \leq |(h^{-1}f-g^{-1}f)(y)|+|g^{-1}f(y)-g^{-1}f(x)| < \|h^{-1}f-g^{-1}f\|+\epsilon/2 < \epsilon,
\]
i.e. $f(hy)\in W$.
Thus,
$U\cdot V\subset f^{-1}(W)$.
\end{proof}

\begin{thm}[Ellis] \label{Ellis}
Let $G$ be a locally compact group and $X$ a locally compact space. Then every separately continuous action of $G$ on $X$ is jointly continuous.
\end{thm}

This is a corollary of Ellis' joint continuity theorem \cite{Ellis-joint}. We give below an independent short proof, assuming that $G$ is first countable.
We will relay on the following well known fact.

\begin{prop} \label{weak->strong}
For a representation of a locally compact group on a Banach space by bounded operators, the following are equivalent:
\begin{itemize}
\item the orbit maps are weakly continuous
\item the orbit maps are strongly continuous.
\end{itemize}
\end{prop}

\begin{proof}
This is a standard approximate identity argument, see for example \cite[Theorem~2.8]{DLG}.
\end{proof}

\begin{proof}[Proof of Theorem~\ref{Ellis} for first countable groups]
In view of Proposition~\ref{weak->strong} and Lemma~\ref{strong->join}, it is enough to show that for $f\in C_0(X)$, the orbit map $g\mapsto gf$ is weakly continuous.
By Riesz' representation theorem every functional on $C_0(X)$ is represented by a finite complex measure and by  the Hahn-Jordan decomposition
it is enough to consider a positive measure $\mu$.
By the first countability of $G$ it is enough to prove that for a converging sequence in $G$, $g_n\to g$, we have the convergence $\int g_nfd\mu \to \int gfd\mu$. This indeed follows from Lebesgue's bounded convergence theorem.
\end{proof}

\subsection{Monoids}

Let $(X,T)$ be a compact semi-topological monoid. By this we mean that 
$X$ is a monoid and $T$ is a compact topology on $X$ for which the product is separately continuous --- for each $y\in X$ the functions 
\[ X\to X,~x \mapsto xy \quad\mbox{and}\quad X\to X,~x\mapsto yx \]
are both continuous, but the map $X\times X \to X$ may not be.
Note that $C(X)$ is invariant under left and right multiplication.
For every $f\in C(X)$ we denote $xf(y)=f(yx)$
and let $S_f$ be the uniform structure obtained on $X$ by pulling back the sup-norm uniform structure from $C(X)$ via the orbit map $X\to C(X)$,
$x\mapsto xf$.
We let $S$ be the uniform structure on $X$ generated by all the structures $S_f$, that is $S=\bigvee_{f\in C(X)} S_f$.


\begin{lem}
The topology $T$ is $S$-compatible.
\end{lem}


\begin{proof}
Note that by Urysohn's lemma $T$ is the weakest topology on $X$ generated by the functions in $C(X)$.
Thus it is enough to show that for a given $f\in C(X)$, the topology $T_f$, generated on $X$ by $f$, is $S$-compatible.
We will show that it is in fact $S_f$-compatible.

Fix $x\in X$ and $\epsilon>0$ and consider $V=f^{-1}(B(f(x),\epsilon))\in T_f$.
Set
\[ V'=f^{-1}(B(f(x),\epsilon/2))\in T_f \quad \mbox{and} \quad U=\{(y,z)~|~\|yf-zf\|<\epsilon/2\} \in S_f. \]
For $y\in UV'$ there exists some $z\in V'$ such that $(y,z)\in U$.
Therefore
\[ |f(y)-f(x)|\leq |yf(e)-zf(e)|+|f(z)-f(x)| < \|yf-zf\|+ \epsilon/2 < \epsilon,\]
and thus $y\in V$.
It follows that $UV'\subset V$.
\end{proof}

Let now $G$ be a locally compact group with a continuous monoid morphism $G\to (X,T)$.
Note that by Theorem~\ref{Ellis} the product map $G\times X\to X$ is continuous.

\begin{lem} 
The action of $G$ on $(X,S)$ is continuous and equicontinuous.
\end{lem}

\begin{proof}
It is enough to show that for every $f\in C(X)$ the action of $G$ on $(X,S_f)$ is continuous and equicontinuous.
Fix $f\in C(X)$.
We first show that the action on $(X,S_f)$ is equicontinuous.
For every $\epsilon>0$ consider the uniformity 
$$
 U=\{(x,y)~|~\|xf-yf\|<\epsilon\}.
$$
Then 
\[ gU=\{(gx,gy)~|~\|xf-yf\|<\epsilon\}=\{(x,y)~|~\|g^{-1}(xf-yf)\|<\epsilon\} = \]
\[ \{(x,y)~|~\|xf-yf\|<\epsilon\}=U,\]
and uniform continuity follows.
We now show that the action is continuous.
By Lemma~\ref{Joint} it is enough to show that for a given $x\in X$ the $x$-orbit map $G\to (X,T_{S_f})$ is continuous.
This is equivalent to showing that the $xf$-orbit map $G\to C(X)$ is strongly continuous,
which is given by Lemma \ref{strong->join}.
\end{proof}

Let us summarize the conclusions of this section:

\begin{thm} \label{conditions}
Let $G$ be a locally compact group and $(X,T)$ a compact semi-topological monoid. Suppose we are given 
a continuous monoid representation $G\to X$ and let $S$ be the associated uniform structure on $X$. 
Then
\begin{itemize}
\item $T$ is an $S$-compatible, 
\item the left regular action $G\act X$ is jointly continuous with respect to both topologies $T$ and $T_S$, and
\item $G\act X$ is equicontinuous with respect to $S$.
\end{itemize}
\end{thm}

\subsection{Weakly almost periodic rigidity}\label{sec:WAP}


%
%

Let $G$ be a locally compact group.
By a {\it monoid representation} of $G$ we mean a continuous monoid homomorphism from $G$ into a compact semi-topological monoid.

\begin{exam}
If $G$ is non-compact we denote by $G^*$ the one-point compactification of $G$, $G\cup \{\infty\}$, with the multiplication extended from that of $G$ by 
$$
 g\infty=\infty g=\infty\infty=\infty
$$ 
for every $g\in G$.
We let $i^*:G\to G^*$ be the obvious embedding.
If $G$ is compact we set $G^*=G$ and $i^*=$ the identity map.
In both cases, $i^*:G\to G^*$ form a monoid representation of $G$.
\end{exam}

We will say that a monoid representation with dense image $i:G\to X$ is {\it universal} if for every
monoid representation $j:G\to Y$ there exists a unique continuous monoid homomorphism $k:X\to Y$ such that $j=ki$. The pair $(i,X)$ will be referred as a {\it universal system}.

\begin{thm} \label{universal}
The locally compact group $G$ admits a universal monoid representation $i:G\to X$.
Every two universal systems are uniquely isomorphic.
Furthermore, $i$ is a homeomorphism into its image and $i(G)$ is open and dense in $X$.
\end{thm}

\begin{proof}
The collection of isomorphism classes of
monoid representations of $G$ with dense images forms a set; it could be described for example
as a subset of the set of all norm closed subalgebras of $C_b(G)$.
Pick one representative for each class and
consider the product space of those, let $i$ be the diagonal morphism from $G$ to this product space
and let $X$ be the closure of $i(G)$.
The existence of $k$ follows immediately. The uniqueness of $k$ follows by the fact that $i(G)$ is dense in $X$, and the uniqueness of the pair $(i,X)$ is obvious.
That $i(G)$ is open follows from the fact that $G^*$ is a factor of $X$.
\end{proof}

\begin{defn}
The representation alluded to in Theorem ~\ref{universal} is called $\WAP(G)$.
\end{defn}

The representation $\WAP(G)$ was defined and studied by de Leeuw and Glicksberg.
Our presentation here is slightly different from theirs.

\begin{rem}
By the Gelfand--Neumark theory, compactifications of $G$ correspond to point separating $*$-subalgebras of $C_b(G)$, where general $*$-subalgebras of the latter correspond to compactifications of (topological) quotients of $G$, and the Stone--\v{C}ech (the largest) compactification correspond to the full algebra $C_b(G)$. Among these, the monoid representations of $G$ correspond to subalgebras carrying an additional structure, and $\WAP(G)$ corresponds to the largest such algebra. It can be shown that it is the algebra of weakly almost periodic functions on $G$, hence the notation.
%
%
We will not elaborate on the point of view of almost periodic functions on $G$.
\end{rem}

\begin{defn}
A group $G$ will be said to be {\it WAP-rigid} if $\WAP(G)\simeq G^*$.
\end{defn}

\begin{exam}
If $G$ is compact then clearly $\WAP(G)=G=G^*$ and $G$ is WAP-rigid.
\end{exam}

The following theorem, which was proved first in \cite{veech} and \cite{EN} (for simple Lie groups), could be seen as a special case of Theorem~\ref{structure} below.
For clarity we give a separate proof.

\begin{thm}[\cite{veech},\cite{EN}] \label{thm:waprigid} \label{7.13}
Let $G$ be an almost simple analytic group over a local field. Then $G$ is WAP-rigid.
\end{thm}

\begin{proof}
We assume $G$ is non-compact.
Let $j:G\to X$ be any monoid representation of $G$.
We will construct a continuous monoid morphism $k:G^*\to X$ satisfying $j=ki^*$.
Such a morphism is clearly unique.
In view of Theorem~\ref{conditions} we are in the situation to apply Theorem~\ref{mainthm}
to either the left or the right actions of $G$ on $X$.
Upon replacing $X$ with $X\times G^*$ we may assume that $j(G)=Ge=eG$ is non-compact.
We therefore get by Theorem~\ref{mainthm} the existence of a point
$x\in \overline{j(G)}$ which is right $G$-invariant and 
a point $y\in \overline{j(G)}$ which is left $G$-invariant. By continuity of the product in $X$ we have $x=xy=y$.
It follows that $x$ is the unique left $G$-invariant point in $\overline{j(G)}$.
We then define $k:G^*\to X$ by setting $k(g)=ge$ for $g\in G$ and $k(\infty)=x$.
Clearly $k$ is a continuous morphism.
\end{proof}

We now discuss semisimple (rather than simple) groups. Let $G$ be a finite centred semisimple analytic group over a local field.
Then $G=G_0G_1\cdots G_n$ where $G_0$ is compact and $G_1,\ldots, G_n$ are the non-compact almost simple factors.
For each $I\subset \{1,\ldots,n\}$ we let 
$$
 G_I=\prod_{i\in I} G_i <G~\text{and}~G^I=G/G_I.
$$
In particular, $G^{\{1,\ldots,n\}}$ is a quotient of $G_0$ hence a compact group.
Note that for $I\subseteq J$ there is a natural homomorphisms $\phi^I_J:G^I\to G^J$.
We denote $\phi_J=\phi^\emptyset_J:G\to G^J$.

We define $\check{G}=\coprod_{I\subseteq \{1,\ldots,n\}} G^I$.
The sets of the form
$$
\bigcup_{J'\subseteq J}(\phi^{J'}_{J})^{-1}(U) \quad\text{and}\quad \check{G}\setminus \bigcup_{J'\subseteq J}(\phi^{J'}_{J})^{-1}(K),
$$
where $J$ is a subset of $\{1,\ldots,n\}$, $U\subset G^J$ is open and $K\subset G^J$ is compact,
generate a compact Hausdorff topology on $\check G$.
We always refer to this topology when regarding $\check{G}$ as a topological space.
In order to understand this topology it might be helpful to note that
for 
$I\subseteq \{1,\ldots,n\}$ and 
a sequence $g_n \in G^I$,
$\check{G}\text{-}\lim g_n=\lim \phi_J(g_n)$ if and only if the right hand side limit, which is the standard limit in the group $G^J$, exists,
where $J$ is the minimal set 
satisfying $I\subseteq J\subseteq \{1,\ldots,n\}$ 
for which $\phi^I_J(g_n)$ is bounded.

We introduce a natural monoid structure on $\check{G}$ as follows.
For $I,J\subseteq \{1,\ldots,n\}$ and $g\in G^I$, $h\in G^J$
we set $gh=\phi_I^{I\cap J}(g)\phi_J^{I\cap J}(h)\in G^{I\cap J}$.
This makes $\check{G}$ a compact semi-topological monoid.

\begin{thm} \label{structure} \label{thm7.14}
$\WAP(G)\simeq \check{G}$.
\end{thm}


\begin{proof}
We prove the theorem by induction on $n$, the number of non-compact simple factors of $G$.
The induction basis is the case $n=1$ which follows from Theorem~\ref{thm:waprigid}.
We let $j:G\to X$ be a monoid representation.
For any $I\subsetneq \{1,\ldots,n\}$ we have by our induction hypothesis $\WAP(G_I)=\check{G}_I$.
In particular $\WAP(G_I)$ has a unique left $G_I$-fixed point which is also a unique right $G_I$-fixed point (as $G_I$ has no compact factor).
It follows that there is a unique left $G_I$-fixed point which is also a unique right $G_I$-fixed point in $\overline{j(G_I)}$.
We denote it by $e_I$.
We define a map $\check{G}\to X$ by sending $g\in G^I$ to $ge_I$.
One checks that this is a continuous morphism.
\end{proof}

\section{WAP representations and mixing} \label{sec:mixing}

Let $k$ be a topological field.
Let $V,V'$ be $k$-vector spaces and $\form{\cdot}{\cdot}:V\times V'\to k$ a bilinear form.
For $v\in V$ and $\phi\in V'$ we denote $\phi(v)=\form{v}{\phi}$.
We assume that the elements of $V'$ separate the points in $V$ and the elements of $V$ separate the points of $V'$.
We denote by $\End(V)^{V'}$ the algebra of endomorphisms $T\in\End(V)$ satisfying for every $\phi\in V'$ that $\phi\circ T$ is represented by an element (necessarily unique) of $V'$, to be denoted $T\phi$.

We endow $V$ with the {\em weak topology}, namely the weakest topology for which every $\phi\in V'$ is a continuous function to $k$.
Note that the elements of $\End(V)^{V'}$ are continuous functions from $V$ to $V$.
Considering the Tychonoff topology on $(V,\mbox{weak})^V$, using the embedding $\End(V) \to V^V$, $T\mapsto (Tv)_v$,
we obtain the {\em weak operator topology} on $\End(V)$, and in particular on $\End(V)^{V'}$.
Check that the composition operation on $\End(V)^{V'}$ is continuous (separately) in each variable, thus $\End(V)^{V'}$ becomes a semi-topological monoid.
Note that $A\subset \End(V)^{V'}$ is precompact if and only if $Av$ is precompact in $V$ for every $v$.

Let $G$ be a topological group.
By a {\em continuous representation} of $G$ to $V$ we mean a continuous monoid homomorphism $\rho:G\to \End(V)^{V'}$.
The representation
$\rho$ is said to be {\em weakly almost periodic}, or WAP, if $\rho(G)$ is precompact in $\End(V)^{V'}$,
or equivalently, if $\rho(G)v$ is precompact in $V$ for every $v\in V$.
In that case, $\overline{\rho(G)}$ is a semi-topological compact monoid.

A representation of $G$ on a Banach space $V$ is said to be a {\em Banaach WAP representaton} if it is WAP with respect to the pairing of $V$ with $V^*$, the space of bounded functionals on $V$.

\begin{exam}
Let $U$ be a Banach space, and consider a strongly continuous homomorphism $G\to \Iso(U)$.
Let $V=U^*$ and $V'=U$, the pairing be the usual one, and the representation $\rho$ be the contragradient representation. By Banach--Alaoglu theorem, $\rho$ is a WAP representation.
A special case of this example is any isometric representation on a reflexive Banach space, and in particular any unitary representation on a Hilbert space.
In these cases $\rho$ is a Banach WAP representation.
\end{exam}

The following is an immediate application of Theorem~\ref{universal}.

\begin{cor}
Let $G$ be a locally compact topological group and let $i:G\to X$ be its universal monoid representation into a compact semi-topological monoid.
Then every WAP-representation $\rho:G\to \End(V)^{V'}$ factors as a representation of $X$, that is there exists a continuous monoid homomorphism
$\rho':X\to \End(V)^{V'}$ such that $\rho=\rho'\circ i$.
\end{cor}

For a locally compact topological group $G$, $\rho:G\to \End(V)^{V'}$ is said to be {\it mixing} if for every $v\in V$, $\phi\in V'$, 
$$
 \lim_{g\to\infty}\form{gv}{\phi}= 0.
$$
Theorem~\ref{structure} gives a structure theorem of representation of semisimple groups.

\begin{thm}\label{thm:GHM}
Let $G$ be a semisimple group and let $\rho:G\to V$ be a WAP representation.
Then $V$ decomposes as a direct sum of representations $V=\bigoplus_{I\subset\{1,\ldots,n\}} V_I$ such that on $V_I$ the $G$-representation factors through $G^I$ and proper factors of $G^I$ have no fixed points in $V_I$.
Furthermore, for every $I$, the representation $V_I$ is $G^I$ mixing.
\end{thm}

A special case of Theorem \ref{thm:GHM} is the classical theorem of Howe--Moore \cite{HM}.

\begin{thm}[{\cite[Theorem~5.1]{HM}}] \label{HM} \label{thm8.4}
Let $G$ be a semisimple analytic group with a finite center (the $F$ point of a Zariski connected semisimple algebraic  group $\mathbb{G}$, defined over a local field $F$) and no compact factor.
Then every ergodic probability preserving action is mixing modulo the action kernel.
\end{thm}

\begin{proof}
Apply the last corollary to the Koopman representation.
\end{proof}


\section{Banach modules}\label{sec:Banach}
We shall now concentrate on the special case of uniformly bounded representations on Banach spaces. The main result of this section, Theorem \ref{dual}, is a  straightforward consequence of Theorem \ref{thm:GHM}, when $G$ is a semisimple group. However, because of the importance of this special case, and for the convenience of the users, we decided to give a self contained discussion that avoids the more general notion of WAP representations. In particular, we shall provide an alternative proof for Theorem \ref{dual}. Since we shall rely in this section only on Theorem \ref{mainthm}, we can state the results for the class of quasi-semisimple groups rather than semisimple groups.
\medskip

Let $V$ be a Banach space and $S$ the norm uniform structure on $V$. We denote by $\mathcal{B}(V)$ the algebra of bounded linear operator on $V$ and by $\GL(V)$ the group of invertibles in $\mathcal{B}(V)$. A group representation $\rho:G\to \GL(V)$ is said to be uniformly bounded if 
$$
 \sup_{g\in G} \|\rho(g)\|_{op} < \infty,
$$
i.e. if it induces a uniform action on $(V,S)$. We denote by $\rho^*:G\to\GL(V^*)$ the dual (contragradient) representation. Since $\|\rho(g)^*\|_{op}=\|\rho(g)\|_{op}$, $\rho^*$ is uniformly bounded iff $\rho$ is.
We will focus on the case where $G$ is a topological group and the representation $\rho$ is continuous with respect to the strong operator topology.

\begin{defn}
We will say that $(V,\rho)$ is a $G$-Banach module if $V$ is a Banach space, $G$ is a topological group and $\rho:G\to \GL(V)$ is a uniformly bounded representation which is continuous in the sense that the map $G\times V\to V,~(g,v)\mapsto \rho(g)(v)$ is continuous. We will say that $(V,\rho)$ is a $G$-Banach $*$-module if also the dual representation $\rho^*:G\to \GL(V^*)$ is continuous in the same sense.
\end{defn}

By Lemma \ref{Joint}, $\rho$ is continuous iff its orbit maps are continuous.
We note that by \cite[Corollary~6.9]{Me98} if $V$ is Asplund (e.g if $V$ is reflexive) and $G$ is an arbitrary topological group then every $G$-Banach module is automoatically a $*$-module.

Apart from the norm topology, $V$ and $V^*$ are equipped with the weak and the weak$^*$ (hereafter $w$ and $w^*$) topologies. It is obvious that these topologies are compatible with the norm uniform structure.
If $G$ is locally compact, it follows by a standard argument of
approximating identity in $L^1(G)$ that a uniformly bounded representation is strongly continuous iff it is weakly continuous, see for example \cite[Theorem~2.8]{DLG}. This is also the case when $V$ is reflexive and $G$ is arbitrary, see \cite{Me01,Me03}.

\begin{thm} \label{dual}
Let $G$ be a quasi-semisimple group.
Let $(V,\rho)$ be a $G$-Banach $*$-module.
Assume that no point in $V^*\setminus\{0\}$ is fixed by a non-compact normal subgroup of $G$.
Then for every $f\in V^*$, 
$$
 \overline{Gf}^{w^*}=Gf\cup\{0\},
$$
and $\rho$ is mixing in the sense that all matrix coefficients tend to 0.
\end{thm}

\begin{proof}
Given $f\in V^*\setminus\{0\}$, consider the space $X=\overline{\conv(Gf)}\setminus\{0\}$.
Let $S$ be the norm uniform structure on $X$ and $T$ the weak*-topology.
By the Hahn--Banach and Alaoglu's theorems $(X,T)$ is locally compact.
By Corollary~\ref{orbits}, $Gf$ is weak*-closed in $X$ and homeomorphic to the coset space $G/G_f$, where the stabiliser $G_f$ is compact. Thus the orbit $Gf$ is non-compact.
It follows that it is not weak*-closed in the compact space $\overline{\conv(Gf)}$, and hence that $\overline{Gf}^{w^*}=Gf\cup\{0\}$. Since the latter is compact while $Gf$ is a proper $G$ space, it follows that $gf\to 0$ (in the weak-$*$ sense) when $g\to\infty$ in $G$.
\end{proof}

\begin{rem}
It follows, for instance, that for a non-compact QSS simple group $G$, the existence of a nonzero invariant vector (or more generally  a vector with a non-compact stabiliser) in a Banach $*$-module $V$ implies the existence of a non-zero invariant vector in $V^*$. This property does not hold for general groups; for example consider the regular representation of a discrete non-amenable group $\gC$ on the space $L^\infty(\gC)$. 
\end{rem}


When $V$ is reflexive, the a priory weaker assumption that $G$ doesn't fix a vector in $V$, is actually sufficient. 

\begin{lem}
Let $L$ be a group and $\rho:L\to\GL(V)$ a uniformly bounded linear representation on a reflexive Banach space $V$. If $L$ has a non-zero invariant vector in $V^*$ then it has a non-zero invariant vector in $V$.
\end{lem}

\begin{proof}
Suppose that $f\in V^*$ is an $L$-invariant norm one functional.
The invariant set of supporting unit vectors
$$
  S_f=\{v\in V:\langle f,v\rangle=\| v\|=1\}
$$
is non-empty by the Hahn--Banach theorem and weakly compact by Alaoglu's theorem.
Hence the Ryll-Nardzewski fixed-point theorem implies that
$L$ admits a fixed point in $S_f$
\end{proof}

\begin{cor}[Howe--Moore's theorem for reflexive Banach spaces, \cite{veech}] \label{reflexive}
Let $G$ be a quasi-semisimple group.
Let $(V,\rho)$ be a reflexive $G$-Banach module.
Assume that no point in $V\setminus\{0\}$ is fixed by a non-compact normal subgroup of $G$.
Then for every $f\in V^*$, $\overline{Gf}^{w^*}=Gf\cup\{0\}$, and $\rho$ is mixing.\end{cor}

We conclude this paper by remarking that for every group $G$, every WAP function on $G$ appears as a matrix coefficient of some reflexive representation.
This result is due to \cite{Me03}, following the important main theorem of \cite{DFJP}.
In this regard, one reverses the logical order and use Corollary~\ref{reflexive} in order to prove
results on WAP compactifications.


\section{Further discussion} \label{sec:hqss}

In writing this paper, our general attitude was to use the axiomatic approach as long as it simplifies and clarifies the discussion, but to keep in mind that the main objects of interest are the classical semisimple groups over local fields.
We held, until now, the temptation of further generalize and axiomatize the results in the price of possibly obscuring their formulation.
We will do this (generalize and obscure) in this section.

\begin{defn}
A locally compact group is called hereditary qss (or hqss) if every non-compact quotient of it is qss.
\end{defn}

It is clear that semisimple groups over local fields are hqss.
We invite the enthusiastic reader to check that (mutatis mutandis) 
Theorem~\ref{thm6.2},
Theorem~\ref{thm6.4},
Theorem~\ref{thm6.6},
Corollary~\ref{cor6.7},
Corollary~\ref{cor6.8},
Corollary~\ref{cor6.9}
and
Theorem~\ref{thm8.4}
are all still valid for the class of hqss group.

Obviously, every simple qss group is hqss (more generally: almost simple qss groups, qss groups for which every proper normal subgroup is precompact),
thus the examples discussed in Theorem~\ref{thm3.6} are hqss.
Note also that the class HQSS consisting of hqss groups is closed under taking products (and quotients).
The reader is also invited to check that Theorem~\ref{7.13} is valid for (almost) simple qss groups and Theorem~\ref{thm7.14} for products of such.

\end{document}